\documentclass[12pt]{article}
\usepackage[a4paper,margin=1.2in]{geometry}
\usepackage{amssymb}
\usepackage{amsthm}
\usepackage{amsmath}
\usepackage{amsfonts}
\usepackage{graphicx}
\usepackage{float}
\usepackage[all]{xy}
\usepackage{scalerel}

\def\hurw{\mathop{\mbox{\textsl{H}}}\nolimits}

\def\exp{\mathop{\mbox{\textsf{exp}}}\nolimits}

\newcommand{\dhurw}{\Delta\hurw}

\newcommand{\Opa}{\mathfrak{A}}

\newcommand{\N}{\mathbb{N}}

\newcommand{\F}{\mathbb{F}}

\newcommand{\binomangle}{\genfrac{\langle}{\rangle}{0pt}{}}

\DeclareMathOperator*{\bcast}{\scalerel*{\circledast}{\sum}}
\DeclareMathOperator*{\bcastq}{\scalerel*{{\bcast}_{q}}{\sum}}

\DeclareMathOperator*{\bbox}{\scalerel*{\boxplus}{\sum}}
\DeclareMathOperator*{\bboxq}{\scalerel*{{\bbox}_{q}}{\sum}}

\newtheorem{theorem}{Theorem}

\newtheorem{definition}{Definition}
\newtheorem{proposition}{Proposition}

\begin{document}
\title{\textbf{One-dimensional dynamical systems type delta over Integral Domains}}
\author{Ronald Orozco L\'opez}

\newcommand{\Addresses}{{

\textit{E-mail address}, R.~Orozco: \texttt{rj.orozco@uniandes.edu.co}
  
}}

\maketitle

\begin{abstract}
In this paper, autonomous differential equations type delta of order one on integral domains are defined. For this we will use the autonomous ring defined on the Hurwitz expansion ring of
exponential generating functions with coefficients in an integral domain. We will also use delta
operators, which behave like derivatives when acting on polynomials, along with Umbral calculus. As 
a particular example of a delta operator we have the forward difference operator that defines the
difference equations. Then the delta-type equations generalize to the ordinary equations and to the
difference equations.
\end{abstract}

\section{Introducci\'on}

En [3] desarrollamos la teor\'ia para poder resolver ecuaciones diferenciales de la forma
\begin{equation}
\phi^{\prime}=f_{1}(\phi)f_{2}(\phi)\cdots f_{n}(\phi)
\end{equation}
en t\'ermino de las soluciones de las ecuaciones m\'as simples $\phi^{\prime}=f_{i}(\phi)$.

Queremos hacer lo mismo con ecuaciones de diferencia de la forma

\begin{equation}
y_{n+1}=f_{1}(y_{n})f_{2}(y_{n})\cdots f_{n}(y_{n})
\end{equation}
esto es, expresar la soluci\'on $y_{n}$ en t\'ermino de las soluciones de las ecuaciones 
$y_{n+1}=f_{i}(y_{n})$. Esto ser\'a muy \'util por ejemplo para resolver la ecuaci\'on de
diferencia log\'istica $y_{n+1}=\mu y_{n}(1-y_{n})$. Esta es la ecuaci\'on en diferencia no
lineal m\'as simple y hasta el momento no tiene una soluci\'on exacta.

Como el operador de derivaci\'on $\delta$ y el operador de diferencia forward $\vartriangle$
son ejemplos de operadores delta, entonces iremos un paso m\'a all\'a definiendo y resolviendo
sistemas din\'amicos uni-dimensionales tipo delta.

Este art\'iculo est\'a dividido de la siguiente forma. Primeramente establecemos conceptos
b\'asicos de sistemas din\'amicos unidimensionales sobre dominios de integridad y conceptos
y resultados de la teor\'ia de operadores delta y del c\'alculo umbral desarrollado por G.C. Rota.
Seguido definimos sistemas din\'amicos unidimensionales tipo delta sobre dominios de integridad y
establecemos algunos resultados b\'asicos sobre estos. Uno de los resultados importantes alcanzado
es que si fijamos la funci\'on $f$ del sistema y variamos los operadores delta obtenemos un grupo
de sistemas din\'amicos. Finalizamos este art\'iculo introduciendo algunos 
sistemas din\'amicos tipo delta cuando delta es el operador de diferencia forward, el operador
de diferencia backward, el operador Abel y el operador Touchard.

En todo este artículo $R$ será un anillo de característica cero. Además, $\N$ será el monoide
de enteros no negativos y $\N_{1}$ será el semigrupo de enteros positivos.

\section{Preliminares}

Iniciamos este art\'iculo presentando la informaci\'on necesaria para construir la teor\'ia de
sistemas din\'amico tipo delta definidos sobre un dominio de integridad de caracter\'istica cero.
Primero son presentados los resultados b\'asicos de [3]. Necesitamos el siguiente isomorfismo de
anillos
\begin{equation}
\dhurw_{R}[[x]]\rightarrow\Opa(\dhurw_{R}[[x]])\rightarrow\rho_{t}\Opa(\dhurw_{R}[[x]])
\end{equation}
en donde $\dhurw_{R}[[x]]$ es el anillo de sucesiones de derivadas de funciones en $\hurw_{R}[[x]]$,
el anillo $\Opa(\dhurw_{R}[[x]])$ es construido usando el operador aut\'onomo definido sobre 
$\hurw_{R}[[x]]$. Este operador es definido usando los polinomios de Bell, por lo tanto exhibe
su car\'acter combinatorial. Finalmente el anillo $\rho_{t}\Opa(\dhurw_{R}[[x]])$ es el anillo
que contiene los semi flujos soluci\'on de las ecuaciones $\phi^{\prime}=f(\phi)$ para cada
$f\in\hurw_{R}[[x]]$. Otros resultados que necesitaremos son obtenido del c\'alculo umbral de Rota,
ve\'ase [4], particularmente los operadores delta que son operadores lineales que generalizan a los operadores
de derivaci\'on.

\subsection{Ecuaciones diferenciales aut\'onomas de orden uno sobre dominios de integridad}

Denote $R$ un dominio de integridad de caracter\'istica cero y sea $\hurw_{R}[[x]]$ el anillo
de Hurwitz, ve\'ase [2], de las funciones generadoras exponenciales de las sucesiones definidas 
sobre $R$ con suma y producto ordinarios de series de potencias. Denote $\delta$ la derivaci\'on
definida sobre $\hurw_{R}[[x]]$ y denote $\dhurw_{R}[[x]]$ el anillo expansi\'on de Hurwitz de 
$\hurw_{R}[[x]]$, esto es, el anillo de sucesiones de derivadas $(f(x),\delta f(x),\delta^{2}f(x),...)$ para cada $f(x)$ en $\hurw_{R}[[x]]$, donde la suma $+$ en $\dhurw_{R}[[x]]$ es definida
componente a componente y el producto $\ast$ es el producto Hurwitz de sucesiones dado por
\begin{equation}
\Delta f(x)\ast\Delta g(x)=\left(\sum_{k=0}^{n}\delta^{k}f(x)\delta^{n-k}g(x)\right)_{n\in\N}.
\end{equation}

Como $\Delta f+\Delta g=\Delta(f+g)$ y $\Delta f\ast\Delta g=\Delta(fg)$, entonces los anillos 
$\dhurw_{R}[[x]]$ y $\hurw_{R}[[x]]$ son isomorfos.
El anillo $\dhurw_{R}[[x]]$ sirve de base para la definici\'on del operador aut\'onomo.
Sea $\hurw_{S}$ el anillo de sucesiones definidas sobre el anillo $S=\hurw_{R}[[x]]$. El 
\textbf{operador aut\'onomo} $\Opa$ es el mapa no lineal
$\Opa:\Delta\hurw_{R}[[x]]\rightarrow\hurw_{S}$ definido por 
\begin{equation}
\Opa(\Delta f)=(A_{n}([\delta^{n-1}f]))_{n\in\N_{1}}
\end{equation}
con $[\delta^{n}f]=(f,\delta f,...,\delta^{n}f)$, en donde los $A_{n}$ son definidos
recursivamente por
\begin{eqnarray}
A_{1}([\delta^{0}f])&=&A_{1}([f])=f,\\
A_{n+1}([\delta^{n}f])&=&Y_{n}(A_{1}([\delta^{0}f]),A_{2}([\delta^{1}f])...,A_{n}([\delta^{n-1}f]);\delta^{1}f,...,\delta^{n}f),
\end{eqnarray}
$n\geq 1$.\\
Los polinomios $A_{n}$ en las indeterminadas $\delta^{0}f,\delta^{1}f,...,\delta^{n-1}f$ 
serán llamados \textbf{polinomios aut\'onomos}.
Denote $\Opa(\dhurw_{R}[[x]])$ el conjunto de todas las sucesiones $\Opa(\Delta f)$. 
Definiendo la suma $\boxplus$ en $\Opa(\dhurw_{R}[[x]])$ de la siguiente manera
\begin{equation}\label{def_bplus}
\Opa(\Delta f)\boxplus\Opa(\Delta g)=\Opa(\Delta f)+\Opa(\Delta g)+\left( H_{n}(f,g)\right)_{n\in\N_{1}}
\end{equation}
en donde $H_{1}(f,g)=0$ y 
\begin{equation}\label{eqn_Hn}
H_{n+1}(f,g)=f\delta A_{n}([\delta^{(n-1)}g])+g\delta A_{n}([\delta^{(n-1)}f])+(f+g)\delta H_{n}(f,g)
\end{equation}
y el producto $\circledast$ en $\Opa(\dhurw_{R}[[x]])$ de la siguiente manera
\begin{equation}\label{def_cdast}
\Opa(\Delta f)\circledast\Opa(\Delta g)=\left(\sum_{l=1}^{n}\mathcal{A}_{l}(f)\mathcal{A}_{n-l+1}(g)\right)_{n\in\N_{1}}
\end{equation}
en donde
\begin{equation}
\mathcal{A}_{l}(f)=\sum_{\vert p(n)\vert=l}\alpha_{\vert p(n)\vert} A_{j_{1}}([\delta^{j_{1}-1}f])A_{j_{2}}([\delta^{j_{2}-1}f])\cdots A_{j_{r}}([\delta^{j_{r}-1}f])
\end{equation}
y $\alpha_{\vert p(n)\vert}$ son n\'umeros apropiados, entonces $(\Opa(\dhurw_{R}[[x]]),\boxplus,\circledast)$ es un anillo conmutativo con unidad. Este anillo ser\'a llamado \textbf{anillo 
aut\'onomo} .

Como $\Opa(\Delta f)\boxplus\Opa(\Delta g)=\Opa(\Delta f+\Delta g)$ y $\Opa(\Delta f)\circledast\Opa(\Delta g)=\Opa(\Delta f\ast\Delta g)$, entonces los anillos $\dhurw_{R}[[x]]$ y 
$\Opa(\dhurw_{R}[[x]])$ son isomorfos.

Ahora mostraremos la relaci\'on entre anillo aut\'onomo y sistemas din\'amicos. Definimos un
\textbf{sistema dinámico} sobre el anillo $R$ como la terna $(R,\hurw_{S}[[t]],\Phi)$
donde $R$ es el \textbf{conjunto de tiempos}, $\hurw_{S}[[t]]$ el \textbf{espacio de fases} y 
$\Phi$ es el \textbf{flujo} 
\begin{equation}\label{flujo_diferencial}
\Phi(t,x,\Delta f(x))=x+\sum_{n=1}^{\infty}A_{n}([\delta^{(n-1)}f(x)])\dfrac{t^{n}}{n!}.
\end{equation}
esto es, $\Phi$ es el mapa definido
por $\Phi:R\times\hurw_{S}[[t]]\times\Delta\hurw_{R}[[x]]\rightarrow\hurw_{S}[[t]]$ donde 
$S=\hurw_{R}[[x]]$.

El flujo $\Phi$ satisface
\begin{enumerate}
\item $\Phi(0,x,\Delta f(x))=x$,
\item $\Phi(t,\Phi(s,x,\Delta f(x)),\Delta f(x))=\Phi(t+s,x\Delta f(x))$,
\item $f(x)\delta_{x}\Phi(t,x,\Delta f(x))=\delta_{t}\Phi(t,x,\Delta f(x))=f(\Phi)$, 
\item Para todo $a\in R$ se cumple que $\Phi(t,x,a\Delta^{}f(x))=\Phi(at,x,\Delta^{}f(x))$.
\end{enumerate}

Las propiedades 1. y 2. le dan a $\Phi$ estructura de grupo abeliano. Este grupo act\'ua sobre 
el espacio de fases $\hurw_{S}[[t]]$. La propiedad 4. dota a $\Phi$ de estructura de $R$-m\'odulo.
Ponga $\Phi_{t,r}(x)\equiv\Phi(t,x,r\Delta f(x))$. Entonces
\begin{equation}
\Phi_{t,r}(x)=\Phi_{rt,1}(x)=\Phi_{r,t}(x)=\Phi_{1,rt}(x)
\end{equation}
Luego podemos definir
\begin{equation}
\Phi_{R,1}(x)=\{\Phi_{t,1}(x):t\in R\}.
\end{equation}

Defina el mapa $\star:R\times\Phi_{R,1}\rightarrow\Phi_{R,1}$ por
$s\star\Phi_{t,1}=\Phi_{st,1}$. Entonces $\Phi_{R,1}$ es un $R$-módulo.

Por otro lado, sea $\rho_{t}$ el mapa llevando $\Opa(\dhurw_{R}[[x]])$ a 
$\rho_{t}\Opa(\dhurw_{R}[[x]])$ definido por
\begin{equation}
\Opa(\Delta f(x))\mapsto\sum_{n=1}^{\infty}A_{n}([\delta^{n-1}f(x)])\frac{t^{n}}{n!}
\end{equation} 

Definimos un \textbf{semi sistema dinámico} como la terna $(R,\hurw_{S}[[t]],\Psi)$
donde $R$ es el conjunto de tiempos, $\hurw_{S}[[t]]$ el espacio de fases, con 
$S=\Opa(\Delta\hurw_{R}[[x]])$ y $\Psi$ es el \textbf{semi flujo} 
\begin{equation}
\Psi(t,x,\Delta f(x))=\Phi(t,x,\Delta f(x))-x=\rho_{t}\Opa(\Delta f(x))
\end{equation}
esto es, $\Psi$ es el mapa
$\Psi:R\times\hurw_{S}[[t]]\times\Delta\hurw_{R}[[x]]\rightarrow\hurw_{S}[[t]]$ .

Luego es posible extender las operaciones del anillo $\Opa(\dhurw_{R}[[x]])$ al conjunto 
$\rho_{t}\Opa(\dhurw_{R}[[x]])$. Tome $\Delta f(x)$ y $\Delta g(x)$ de $\dhurw_{R}[[x]]$. Definimos
la suma $\boxplus$ y el producto $\circledast$ en $\rho_{t}\Opa(\dhurw_{R}[[x]])$ así
\begin{eqnarray}
\rho_{t}\Opa(\Delta f)\boxplus\rho_{t}\Opa(\Delta g)&=&\sum_{n=1}^{\infty}\left(A_{n}([\delta^{n-1}f(x)])+A_{n}([\delta^{n-1}f(x)])+H_{n}(f,g)\right)\\
\rho_{t}\Opa(\Delta f)\circledast\rho_{t}\Opa(\Delta g)&=&\sum_{n=1}^{\infty}\left(\sum_{l=1}^{n}\mathcal{A}_{l,n}(f)\mathcal{A}_{n-l+1,n}(g)\right)\frac{t^{n}}{n!}
\end{eqnarray}

Entonces $\rho_{t}\Opa(\dhurw_{R}[[x]])$ con la suma $\boxplus$ y el producto $\circledast$ es 
un anillo conmutativo con unidades $0$ y $\rho_{t}\Opa(\textbf{e}_{\Delta})=t$. Claramente
los anillos $\Opa(\dhurw_{R}[[x]])$ y $\rho_{t}\Opa(\dhurw_{R}[[x]])$ son isomorfos. Luego podemos
factorizar semi flujos.
Sea $f=g_{1}g_{2}\cdots g_{m}$ un producto de funciones en $\hurw_{R}[[x]]$. Entonces
\begin{equation}
\rho_{t}\Opa(\Delta f)=\rho_{t}\Opa(\Delta g_{1})\circledast\cdots\circledast\rho_{t}\Opa(\Delta g_{m})
\end{equation}

implica que 
\begin{equation}
\Phi(t,x,\Delta f(x))=x+(\Phi(t,x,\Delta g_{1}(x))-x)\circledast\cdots\circledast(\Phi(t,x,\Delta g_{m}(x))-x)
\end{equation}
De este modo podemos expresar el flujo de un sistema din\'amico en t\'erminos de producto de
semi flujos. Lo cual es muy \'util para encontrar soluciones de ecuaciones diferenciales en 
t\'ermino de soluciones de ecuaciones m\'as simples.

\subsection{C\'alculo finito de operadores}

Sea $\F$ un campo de caracter\'istica cero y $\F[t]$ el $\F$-\'algebra de polinomios con coeficientes 
en $\F$ en la variable $t$. 

Una \textbf{sucesi\'on polinomial} $(p_{n}(t))_{n\in\N}$ en $\F[t]$ es un sucesi\'on de polinomios 
tal que $p_{n}(t)$ tiene exactamente grado $n$. Una sucesi\'on polinomial se dice de
\textbf{tipo binomial} si se cumple
\begin{equation}
p_{n}(t+s)=\sum_{k=0}^{n}\binom{n}{k}p_{k}(t)p_{n-k}(s)
\end{equation}
para todo $n\in\N$.

Si $a$ es un elemento de $\F$, entonces $E^{a}$ es el operador shift enviando un polinomio $p(t)$ 
a $p(t+a)$. Se puede mostrar que 
\begin{equation}
E^{a}=e^{a\delta}=\sum_{k=0}^{\infty}\frac{a^{k}}{k!}\delta^{k}
\end{equation}

Un operador $Q$ es \textbf{shift invariante} si $QE^{a}=E^{a}Q$ para cualquier elemento $a$ en $\F$.
El operador $\delta$ es claramente un operador shift invariante. Denote $\hurw_{R}[[\delta]]$
el anillo de Hurwitz de operadores shift invariantes.

Un operador $Q:\F[t]\rightarrow\F[t]$ es un \textbf{operador delta} si $Q$ es shift invariante y 
$Qt=c$, donde $c$ es una constante no cero. Una sucesi\'on polinomial $(p_{n}(t))_{n\in\N}$
es un \textbf{conjunto base} para el operador delta $Q$ si $p_{0}(t)=1$, $p_{n}(0)=0$ para todo 
$n\geq1$ y $Qp_{n}(t)=np_{n-1}(t)$.

Algunas propiedades de operadores delta son:
\begin{enumerate}
\item Si $p(t)$ es un polinomio de grado $n$, entonces $Qp(t)$ es un polinomio de grado $n-1$.
\item Sean $(p_{n}(t))_{n\in\N}$ una sucesi\'on de polinomios y $Q:\F[t]\rightarrow\F[t]$ un operador
definido por $Qp_{n}(t)=np_{n-1}(t)$. Entonces son equivalente
\begin{enumerate}
\item La sucesi\'on $(p_{n}(t))_{n\in\N}$ es de tipo binomial
\item El operador $Q$ es un operador delta
\item El operador $Q$ tiene la expansión formal
\begin{equation}
Q=\sum_{k=0}^{\infty}a_{k}\frac{\delta^{k}}{k!}
\end{equation}
es decir, $Q\in\hurw_{R}[[\delta]]$.
\end{enumerate}
\end{enumerate}

Lo anterior significa entre otras cosas que las sucesiones polinomiales de tipo binomial son 
caracterizadas por operadores delta. Los siguientes son resultados importantes en la teor\'ia
de operadores delta:

\textbf{Primer teorema de expansi\'on}. Sea $T$ un operador shift invariante y $Q$ un operador delta
con conjunto base $(p_{n}(x))$. Entonces
\begin{equation}
T=\sum_{k=0}^{\infty}\frac{[Tp_{k}(x)]_{x=0}}{k!}Q^{k}
\end{equation}

\textbf{Teorema de isomorfismo}. Sea $Q$ un operador delta con conjunto base $(q_{n}(x))_{n\in\N}$.
Entonces el mapa $\hurw_{R}[[\delta]]\rightarrow\hurw_{R}[[t]]$ dado por
\begin{equation}
T\mapsto\sum_{k=0}^{\infty}[Tq_{k}(x)]_{x=0}\frac{t^{k}}{k!}
\end{equation}
es un isomorfismo de anillos.

Sea $Q$ un operador delta con conjunto base $p_{n}(t)$ y sea $Q=q(\delta)$. Sea $q^{-1}(t)$ la
serie de potencias formal inversa. Entonces 
\begin{equation}\label{eqn_exp_inver}
\sum_{n=0}^{\infty}\frac{q_{n}(t)}{n!}u^{n}=e^{tq^{-1}(u)}.
\end{equation} 

Un operador $L:\F[t]\rightarrow\F[t]$ es un \textbf{operador umbral} si existen dos sucesiones
bases $(p_{n}(t))_{n\in\N}$ y $(q_{n}(t))_{n\in\N}$ tal que $Lp_{n}(t)=q_{n}(t)$, para todo $n\in\N$.
Existe un \'unico operador lineal $L$ actuando sobre $\F[t]$ tal que $L(t^{n})=q_{n}(t)$. Este 
operador $L$ es conocido como la \textbf{representaci\'on umbral} de la sucesi\'on $q_{n}(t)$. 
Si
\begin{equation}
p_{n}(t)=\sum_{k=0}^{n}a_{n,k}t^{k},
\end{equation}
entonces la \textbf{composici\'on umbral} es la sucesi\'on $(r_{n}(t))_{n\in\N}$ definida por
\begin{equation}
r_{n}(t)=\sum_{k=0}^{n}a_{n,k}q_{k}(t)
\end{equation}
es decir $r_{n}(t)=Lp_{n}(t)$, donde $L$ es la representaci\'on umbral de $(q_{n}(t))_{n\in\N}$.
Usaremos la notaci\'on $r_{n}(t)=p_{n}(\underline{q}(t))$ para $r_{n}(t)=Lp_{n}(t)$. Luego
definimos la composici\'on $(p\circ q)(t)$ como la sucesi\'on $(p_{n}(\underline{q}(t)))_{n\in\N}$.
Denote $\mathcal{U}$ el conjunto de sucesiones de base. Entonces $\mathcal{U}$ es un grupo de Lie con
la composici\'on umbral $\circ$ como operaci\'on en donde dos sucesiones $(p_{n}(t))_{n\in\N}$
y $(q_{n}(t))_{n\in\N}$ son inversas si $p_{n}(\underline{q}(t))=t^{n}$. Si $p_{n}(t)$ y $q_{n}(t)$
son sucesiones bases con operadores delta $P=f(\delta)$ y $Q=g(\delta)$, entonces 
$p_{n}(\underline{q}(t))$ es un conjunto base con operador delta $f(g(\delta))$. Luego 
$(p_{n}(t))_{n\in\N}$ y $(q_{n}(t))_{n\in\N}$ son inversas si y s\'olo si $f(g(x))=t$.

\section{Sistemas Dinámicos tipo Delta}

Sea $Q$ un operador delta con conjunto base $q=(q_{n}(t))_{n\in\N}$. Ahora sea $L$ la 
representaci\'on umbral de los polinomios base $q$, esto es, $L(t^{n})=q_{n}(t)$. El operador delta
se comporta como una derivada cuando actua sobre polinomios. Lo que queremos hacer es definir
una nueva clase de ecuaciones diferenciales en donde en vez de usar una derivada ordinaria usaremos
operadores delta y luego encontrar analog\'ias en los resultados encontrados en sistemas din\'amicos
unidimensionales definidos sobre dominios. 

\begin{theorem}
Denote $\Phi_{Q}(t,x,\Delta f(x))=L[\Phi(t,x,\Delta f(x))]$. Entonces
\begin{equation}
Q\Phi_{Q}(t,x,\Delta f(x))=f(\Phi_{Q}(t,x,\Delta f(x))).
\end{equation}
\end{theorem}
\begin{proof}
Usamos la definici\'on de $\Phi(t,x,\Delta f(x))$ y aplicamos $QL$. Tenemos
\begin{eqnarray*}
QL[\Phi(t,x,\Delta f(x))]&=&QL\left(x+\sum_{k=1}^{\infty}A_{k}([\delta^{k-1}f(x)])\frac{t^{k}}{k!}\right)\\
&=&xQL(1)+\sum_{k=1}^{\infty}A_{k}([\delta^{k-1}f(x)])\frac{QL(t^{k})}{k!}\\
&=&\sum_{k=1}^{\infty}A_{k}([\delta^{k-1}f(x)])\frac{kL(t^{k-1})}{k!}\\
&=&\sum_{k=0}^{\infty}A_{k+1}([\delta^{k}f(x)])\frac{L(t^{k})}{k!}\\
&=&L\left(\sum_{k=0}^{\infty}A_{k+1}([\delta^{k}f(x)])\frac{t^{k}}{k!}\right)\\
&=&L\left(f(x)\delta_{x}\Phi(t,x,\Delta f(x))\right)
\end{eqnarray*}
Por las propiedades de flujo la \'ultima l\'inea arriba da $L[f(\Phi(t,x,\Delta f(x)))]$. Finalmente
poniendo $\Phi_{Q}(t,x,\Delta f(x))=L[\Phi(t,x,\Delta f(x))]$ obtenemos el resultado deseado.
\end{proof}

Luego de este resultado definimos ecuaciones diferenciales aut\'onomas de orden uno con 
derivaci\'on delta de este modo

\begin{definition}
Sea $Q$ un operador umbral. Tome $f(x)$ en $\hurw_{R}[[x]]$. Una ecuaci\'on diferencial aut\'onoma 
uni-dimensional \textbf{tipo delta} es una ecuaci\'on de la forma
\begin{equation}
Q\Phi=f(\Phi).
\end{equation}
\end{definition}

Esta definici\'on generaliza a las ecuaciones que ya conocemos. Cuando $Q=\delta$ obtenemos 
las ecuaciones diferenciales ordinarias. Si $Q$ es el operador de diferencia $\vartriangle$, entonces
obtenemos las ecuaciones en diferencias de orden uno. Ahora bien esta definición permite obtener
nuevos tipos de ecuaciones según el operador delta $Q$ que usemos. Si $Q=\triangledown$, obtenemos
\textbf{ecuaciones de diferencia backward}. Si $Q=E^{a}\delta$, entonces obtenemos 
\textbf{ecuaciones de tipo Abel}. Si $Q=\log(I+\delta)$, entonces obtenemos \textbf{ecuaciones 
de tipo Touchard}. Las soluciones a este tipo de soluciones ser\'an estudiadas m\'as adelante.
Ahora queremos encontrar una soluci\'on que generalice a las soluciones encontradas en secciones
anteriores.

Sea $Q$ un operador delta con conjunto base $q=(q_{n}(t))_{n\in\N}$ con representaci\'on umbral
$L$. Tome $\Opa(\Delta f(x))$ en 
$\Opa(\dhurw_{R}[[x]])$ y denote $\rho_{q}$ el mapa definido por
\begin{equation}
\rho_{q}(\Opa(\Delta f(x)))=\sum_{n=1}^{\infty}A_{k}([\delta^{n-1}f(x)])\frac{q_{n}(t)}{n!}.
\end{equation}
Luego $\rho_{q}\Opa(\Delta f(x))=L(\rho_{t}\Opa(\Delta f(x)))$ y $L$ es la representaci\'on
umbral de $\rho_{q}\Opa(\Delta f(x))$.

\begin{definition}
Sea $Q$ un operador delta con conjunto base $q=(q_{n}(t))_{n\in\N}$. Tome $\Delta f(x)$ y 
$\Delta g(x)$ de $\dhurw_{R}[[x]]$. Definimos la suma $\boxplus_{q}$ y el
producto $\circledast_{q}$ en $\rho_{q}\Opa(\dhurw_{R}[[x]])$ así
\begin{eqnarray}
\rho_{q}\Opa(\Delta f)\boxplus_{q}\rho_{q}\Opa(\Delta g)&=&\sum_{n=1}^{\infty}(A_{n}([\delta^{n-1}f(x)])+A_{n}([\delta^{n-1}f(x)])\nonumber\\
&&+H_{n}(f,g))\frac{q_{n}(t)}{n!}
\end{eqnarray}
donde 
$$H_{n+1}(f,g)=f\delta A_{n}([\delta^{n-1}g])+g\delta A_{n}([\delta^{n-1}f])+(f+g)\delta H_{n}(f,g)$$
y
\begin{eqnarray}
\rho_{q}\Opa(\Delta f)\circledast_{q}\rho_{q}\Opa(\Delta g)&=&\sum_{n=1}^{\infty}\left(\sum_{l=1}^{n}\mathcal{A}_{l,n}(f)\mathcal{A}_{n-l+1,n}(g)\right)\frac{q_{n}(t)}{n!}
\end{eqnarray}
\end{definition}

Ahora dotamos de estructura de anillo al conjunto $\rho_{q}\Opa(\dhurw_{R}[[x]])$

\begin{theorem}
Sea $Q$ un operador delta con conjunto base $q=(q_{n}(t))_{n\in\N}$. 
El conjunto $\rho_{q}\Opa(\dhurw_{R}[[x]])$ con la suma $\boxplus_{q}$ y el producto 
$\circledast_{q}$ es un anillo conmutativo con unidades $\rho_{q}\Opa(0)=0$ y 
$\rho_{q}\Opa(\textbf{e}_{\Delta})=q_{1}(t)=t$.
\end{theorem}
\begin{proof}
De (\ref{def_bplus}) y (\ref{def_cdast}) se sigue que
\begin{eqnarray*}
\rho_{q}\Opa(\Delta f)\boxplus_{q}\rho_{q}\Opa(\Delta g)&=&\rho_{q}[\Opa(\Delta f)\boxplus\Opa(\Delta g)]\\
\rho_{q}\Opa(\Delta f)\circledast_{q}\rho_{q}\Opa(\Delta g)&=&\rho_{q}[\Opa(\Delta f)\circledast\Opa(\Delta g)]
\end{eqnarray*}
Luego la prueba sigue igual a la que se hizo en [?].
\end{proof}

As\'i el mapa $\rho_{q}$ es un isomorfismo de anillos.

Denote $L\hurw_{R}[[x]]$ la imagen de $\hurw_{R}[[x]]$ por el operador umbral. Extenderemos la
definici\'on de semi sistemas din\'amicos a operadores delta

\begin{definition}
Sea $Q$ un operador delta con conjunto base $q=(q_{n}(t))_{n\in\N}$. 
Definimos un \textbf{semi sistema dinámico tipo delta} como la terna $(R,L\hurw_{S}[[t]],\Psi_{Q})$
donde $R$ es el conjunto de tiempos, $L\hurw_{S}[[t]]$ el espacio de fases, con 
$S=\Opa(\Delta\hurw_{R}[[x]])$ y $\Psi_{Q}$ es el \textbf{semi flujo delta} 
\begin{equation}
\Psi_{Q}(t,x,\Delta f(x))=\rho_{q}\Opa(\Delta f(x))
\end{equation}
esto es, $\Psi_{Q}$ es el mapa
$\Psi_{Q}:R\times L\hurw_{S}[[t]]\times\Delta\hurw_{R}[[x]]\rightarrow L\hurw_{S}[[t]]$ .
\end{definition}

Esta definici\'on junto con que $\rho_{q}\Opa(\dhurw_{R}[[x]])$ es un anillo nos lleva a la siguiente
definici\'on

\begin{definition}
Sea $Q$ un operador delta con conjunto base $q=(q_{n}(t))_{n\in\N}$. 
Definimos el anillo de semi sistemas din\'amicos tipo delta como el conjunto
\begin{equation}
(R,L\hurw_{S}[[x]],\rho_{q}\Opa(\dhurw_{R}[[x]]))=\{(R,L\hurw_{S}[[x]],\Psi_{Q}):\Psi_{Q}=\rho_{q}\Opa(\Delta f)\}
\end{equation}
donde 
\begin{small}
\begin{eqnarray*}
(R,L\hurw_{S}[[x]],\rho_{q}\Opa(\Delta f))\boxplus_{q}(R,L\hurw_{S}[[x]],\rho_{q}\Opa(\Delta g))&=&(R,L\hurw_{S}[[x]],\rho_{q}\Opa(\Delta(f+g)))
\end{eqnarray*}
y
\begin{eqnarray*}
(R,L\hurw_{S}[[x]],\rho_{q}\Opa(\Delta f))\circledast_{q}(R,L\hurw_{S}[[x]],\rho_{q}\Opa(\Delta g))&=&(R,L\hurw_{S}[[x]],\rho_{q}\Opa(\Delta(f\ast g)))
\end{eqnarray*}
\end{small}
y en donde $(R,L\hurw_{S}[[x]],0)$ y $(R,L\hurw_{S}[[x]],t)$ son las unidades con respecto a 
$\boxplus_{q}$ y $\circledast_{q}$ respectivamente.
\end{definition}

Como veremos m\'as adelante el anillo $(R,L\hurw_{S}[[x]],\rho_{q}\Opa(\dhurw_{R}[[x]]))$ contiene
todas las soluciones a las ecuaciones tipo delta $Q\Phi=f(\Phi)$ para cada 
función $f(x)$ en $\hurw_{R}[[x]]$. Gracias a este anillo será posible descomponer las soluciones 
de una ecuación diferencial aut\'onoma tipo delta de orden uno en soluciones m\'as simples.\\ 

Un tipo importante de ecuaciones tipo delta se dan cuando $f(x)$ es un polinomio. Si $f(x)$
es separable, entonces podemos factorizar la soluci\'on de $Q\Phi=f(\Phi)$ usando la factorizaci\'on
de $\rho_{q}\Opa(\Delta f(x))$ en el anillo $\rho_{q}\Opa(\dhurw_{R}[[x]])$. Primero calcularemos 
$\rho_{q}\Opa(\Delta(ax+b))$ 

\begin{proposition}\label{prop_pol_lin}
Sea $Q$ un operador delta con conjunto base $q=(q_{n}(t))_{n\in\N}$ y sea $Q=p(\delta)$. Entonces
\begin{equation}
\rho_{q}\Opa(\Delta(ax+b))=\frac{1}{a}(ax+b)\left(e^{tp^{-1}(a)}-1\right),
\end{equation}
$a\neq0$ y en donde $p^{-1}$ es la serie de potencia formal inversa de $p$.
\end{proposition}
\begin{proof}
Aplicamos los operadores $\Delta$, $\Opa$ y $\rho_{q}$ a la funci\'on $f(x)=ax+b$. As\'i tenemos
\begin{eqnarray*}
\rho_{q}\Opa(\Delta(ax+b))&=&\rho_{q}\Opa(ax+b,a,0,...)\\
&=&\rho_{q}(ax+b,a(ax+b),a^{2}(ax+b),...)\\
&=&\sum_{n=1}^{\infty}a^{n-1}(ax+b)\frac{q_{n}(t)}{n!}\\
&=&\frac{1}{a}(ax+b)\sum_{n=1}^{\infty}a^{n}\frac{q_{n}(t)}{n!}\\
&=&\frac{1}{a}(ax+b)\left(e^{tp^{-1}(a)}-1\right)
\end{eqnarray*}
en donde hemos usado (\ref{eqn_exp_inver}).
\end{proof}

Ahora si $f(x)$ factoriza como $(a_{1}x+b_{1})(a_{2}x+b_{2})\cdots(a_{n}x+b_{n})$ entonces
tenemos

\begin{theorem}\label{theo_poly_prod}
Sea $Q$ un operador delta con conjunto base $q=(q_{n}(t))_{n\in\N}$ y sea $Q=p(\delta)$. Suponga
$f(x)=\prod_{k=1}^{n}(a_{k}x+b_{k})$ un polinomio separable en $\F[x]\subset\hurw_{\F}[[x]]$.
Entonces
\begin{equation}
\rho_{q}\Opa(\Delta f(x))=\bcastq_{k=1}^{n}\frac{1}{a_{k}}(a_{k}x+b_{k})\left(e^{tp^{-1}(a_{k})}-1\right)
\end{equation}
en donde $p^{-1}$ es la serie de potencia formal inversa de $p$.
\end{theorem}
\begin{proof}
Por aplicaci\'on directa de la Proposici\'on \ref{prop_pol_lin}.
\end{proof}

Ahora mostraremos otra forma de calcular $Q\Phi=f(\Phi)$. Suponga que 
$f(x)=\sum_{k=0}^{n}a_{k}x^{k}$ es un polinomio de grado $n$ en $R[x]\subset\hurw_{R}[[x]]$.
Queremos calcular el polinomio $\rho_{q}\Opa(\Delta f(x))$. Usaremos la siguiente notaci\'on
\begin{eqnarray}\label{eqn_pot_rho}
\rho_{q}\Opa[\Delta x]_{\circledast_{q}}^{k}&=&\rho_{q}\Opa[\Delta x]\circledast_{q}\cdots\circledast_{q}\rho_{q}\Opa[\Delta x]
\end{eqnarray}

Primero calcularemos los monomios $\rho_{q}\Opa[\Delta x]_{\circledast_{q}}^{k}$

\begin{proposition}\label{lemma_potencia}
Sea $Q$ un operador delta con conjunto base $q=(q_{n}(t))_{n\in\N}$ y sea $Q=p(\delta)$. Sea 
$L$ la representaci\'on umbral de $q_{n}(t)$. Cuando $k=0$ tenemos que 
$\rho_{q}\Opa[\Delta x]_{\circledast_{q}}^{0}=at$ y cuando $k=1$, claramente
$\rho_{q}\Opa[\Delta x]_{\circledast_{q}}^{1}=x(e^{tp^{-1}(a)}-1)$. Para $k\geq2$ se tiene
\begin{equation}
\rho_{q}\Opa[\Delta x]_{\circledast_{q}}^{k}=L\left(\frac{x}{\sqrt[k-1]{1-a(k-1)x^{k-1}t}}-x\right)
\end{equation}
\end{proposition}
\begin{proof}
En [?] fue mostrado que 
\begin{equation*}
\rho_{t}\Opa[\Delta x]_{\circledast}^{k}=\frac{x}{\sqrt[k-1]{1-a(k-1)x^{k-1}t}}-x.
\end{equation*}
Como $L$ es la representaci\'on umbral de $q_{n}(t)$, entonces 
$\rho_{q}\Opa[\Delta x]_{\circledast}^{k}=L(\rho_{t}\Opa[\Delta x]_{\circledast}^{k})$
y todo sigue de aqu\'i.
\end{proof}

Por aplicación directa de esta Proposici\'on tenemos el siguiente resultado

\begin{theorem}\label{theo_poly_sum}
Sea $Q$ un operador delta con conjunto base $q=(q_{n}(t))_{n\in\N}$ y sea $Q=p(\delta)$. Suponga
$f(x)=\sum_{k=1}^{n}a_{k}x^{k}$ un polinomio $R[x]\subset\hurw_{R}[[x]]$.
Entonces
\begin{eqnarray}
&&\rho_{q}\Opa(\Delta f(x))=\bboxq_{k=0}^{n}\rho_{q}\Opa[a_{k}\Delta x]_{\circledast_{q}}^{k}\nonumber\\
&&=a_{0}t\boxplus_{q}x(e^{tp^{-1}(a_{1})}-1)\boxplus_{q}\bboxq_{k=2}^{n}L\left(\frac{x}{\sqrt[k-1]{1-a_{k}(k-1)x^{k-1}t}}-x\right)
\end{eqnarray}
\end{theorem}

Finalmente damos la definici\'on de sistemas din\'amicos tipo delta

\begin{definition}
Sea $Q$ un operador delta con conjunto base $q=(q_{n}(t))_{n\in\N}$ y denote $L$ la representaci\'on
umbral de $q_{n}(t)$. Definimos un \textbf{sistema dinámico tipo delta} sobre el anillo $R$ como 
la terna $(R,L\hurw_{S}[[t]],\Phi_{Q})$ donde $R$ es el \textbf{conjunto de tiempos}, 
$L\hurw_{S}[[t]]$ el \textbf{espacio de fases} y $\Phi_{Q}$ es el \textbf{flujo delta} 
\begin{equation}\label{flujo_delta}
\Phi_{Q}(t,x,\Delta f(x))=x+\sum_{n=1}^{\infty}A_{n}([\delta^{(n-1)}f(x)])\dfrac{q_{n}(t)}{n!}.
\end{equation}
esto es, $\Phi_{Q}$ es el mapa definido
por $\Phi_{Q}:R\times L\hurw_{S}[[t]]\times\Delta\hurw_{R}[[x]]\rightarrow L\hurw_{S}[[t]]$ donde 
$S=\hurw_{R}[[x]]$.
\end{definition}

De este modo 

\begin{equation}
\Phi_{Q}(t,x,\Delta f(x))=x+\rho_{q}\Opa(\Delta f(x))
\end{equation}
y junto con los Teoremas \ref{theo_poly_prod} y \ref{theo_poly_sum} encontramos la soluci\'on a la
ecuaci\'on $Q\Phi=f(\Phi)$ cuando $f(x)$ es un polinomio.

La representaci\'on umbral $L$ de $\Phi_{Q}(t,x,\Delta f(x))$ implica que

\begin{enumerate}
\item $\Phi_{Q}(0,x,\Delta f(x))=0$
\item $\Phi_{Q}(t+s,x,\Delta f(x))=\Phi_{Q}(t,\Phi_{Q}(s,x,\Delta f(x)),\Delta f(x))$
\item $Q\Phi_{Q}(t,x,\Delta f(x))=f(x)\delta_{x}\Phi_{Q}(t,x,\Delta f(x))=f(\Phi_{Q})$.
\end{enumerate}

Podemos reescribir (\ref{flujo_delta}) de la siguiente forma
\begin{equation}\label{eqn_delta_rep}
\Phi_{Q}(t,x,\Delta f(x))=\sum_{n=0}^{\infty}Q^{n}\Phi_{Q}(0,x,\Delta f(x))\frac{q_{n}(t)}{n!}
\end{equation}
Llamamos a (\ref{eqn_delta_rep}) una \textbf{delta representaci\'on formal} de 
$\Phi_{Q}(t,x,\Delta f(x))$. 

El \'ultimo resultado de esta secci\'on es mostrar que el conjunto de todos los flujos tipo delta
forman un grupo

\begin{theorem}
Fije una $f(x)$ en $\hurw_{R}[[x]]$. Denote $\mathcal{F}$ el conjunto de flujos tipo delta $\Phi_{Q}$
soluci\'on de $Q\Phi_{Q}=f(\Phi_{Q})$ para todo $Q\in\hurw_{R}[[\delta]]$. Tome operadores delta 
$A,B\in\hurw_{R}[[\delta]]$ con polinomios base $(a_{n}(t))_{n\in\N}$ y $(b_{n}(t))_{n\in\N}$,
respectivamente. Defina la composici\'on $\circledcirc$ de flujos deltas $\Phi_{A}$ y $\Phi_{B}$ 
por $\Phi_{A}\circledcirc\Phi_{B}=\Phi_{C}$, donde
\begin{equation}
\Phi_{C}=x+\sum_{n=1}^{\infty}A_{n}([\delta^{n-1}f(x)])\frac{c_{n}(t)}{n!}.
\end{equation}
y $c_{n}(t)=a_{n}(\underline{b}(t))$. Entonces $(\mathcal{F},\circledcirc)$ es un grupo de Lie
con unidad $\Phi$.
\end{theorem}
\begin{proof}
Como $(c_{n}(t))_{n\in\N}$ es una sucesi\'on de conjuntos base, entonces $\Phi_{C}\in\mathcal{F}$ y
de aquí el producto $\circledcirc$ es cerrado en $\mathcal{G}$. Si $\Phi_{A}\circledcirc\Phi_{B}=\Phi$, entonces $c_{n}(t)=t^{n}$ implica que $(a_{n}(t))$ y $(b_{n}(t))$ son inversos con la 
composici\'on umbral. Luego $\Phi_{A}$ y $\Phi_{B}$ son inversos con respecto a $\circledcirc$.
La asociatividad de $\circledcirc$ sigue de la asociatividad de la composici\'on de umbral $\circ$.
De este modo $(\mathcal{F},\circledcirc)$ es un grupo. Ahora defina el mapa 
$\xi:\mathcal{F}\rightarrow\mathcal{U}$ por $\xi(\Phi_{A})=(a_{n}(t))_{n\in\N}$. Entonces
$\xi(\Phi_{A}\circledcirc\Phi_{B})=(a_{n}(t))_{n\in\N}\circ(b_{n}(t))_{n\in\N}$. As\'i
$\xi$ es un isomorfismo de grupos y por tanto $\mathcal{F}$ es un grupo de Lie.
\end{proof}

Ahora hallaremos la representaci\'on en serie de potencias formal de $\Phi_{Q}(t,x,\Delta f(x))$.
Tendremos presente que $q_{n}(t)=\sum_{i=1}^{n}\beta_{n,i}t^{i}$.

\begin{theorem}\label{theo_matriz}
Sea $Q$ un operador delta con conjunto base $q=(q_{n}(t))_{n\in\N}$ con 
$q_{n}(t)=\sum_{k=1}^{n}\beta_{k,n}t^{k}$ para todo $n\geq1$. Entonces
\begin{equation}\label{eqn_rel}
\Phi_{Q}(t,x,\Delta f(x))=x+\sum_{n=1}^{\infty}\textbf{B}_{n}[\Opa(\Delta f)\cdot F]t^{n}
\end{equation}
adem\'as
\begin{equation}\label{eqn_rel_sum}
\rho_{q}\Opa(\Delta f)\boxplus_{q}\rho_{q}\Opa(\Delta g)=\sum_{n=1}^{\infty}\textbf{B}_{n}[\Opa(\Delta f)\boxplus\Opa(\Delta g)]\cdot Ft^{n}
\end{equation}
y
\begin{equation}\label{eqn_rel_prod}
\rho_{q}\Opa(\Delta f)\circledast_{q}\rho_{q}\Opa(\Delta g)=\sum_{n=1}^{\infty}\textbf{B}_{n}[\Opa(\Delta f)\circledast\Opa(\Delta g)]\cdot Ft^{n}
\end{equation}
donde $\textbf{B}=(\textbf{B}_{n})_{n\in\N}$ es la \textbf{matriz de conexi\'on} conectando 
$(q_{n}(t))$ y $(t^{n})$ 
\begin{equation}
\textbf{B}_{n}=(b_{n,i})_{n\in\N_{1}}
\end{equation}
\begin{equation}
b_{n,i}=
\begin{cases}
0, & i<n,\\
\beta_{n,i}, & i\geq n
\end{cases}
\end{equation}
y $F=(1,\frac{1}{2!},\frac{1}{3!},...)$ y $\cdot$ es el producto de Hadamard de vectores, esto es,
producto componente a componente 
\end{theorem}
\begin{proof}
Probaremos (\ref{eqn_rel}) solamente, pues una prueba para (\ref{eqn_rel_sum}) y (\ref{eqn_rel_prod}) 
son an\'aloga. Pongamos $A_{n}\equiv A_{n}([\delta^{n-1}f])$. Tenemos
\begin{eqnarray*}
\rho_{q}\Opa(\Delta f)&=&\sum_{n=1}^{\infty}A_{n}(f)\frac{1}{n!}\sum_{k=1}^{n}\beta_{k,n}t^{k}\\
&=&A_{1}(f)\beta_{1,1}t+A_{2}(f)(\beta_{1,2}t+\beta_{2,2}t^{2})\frac{1}{2!}\\
&&+A_{3}(f)(\beta_{1,3}t+\beta_{2,3}t^{2}+\beta_{3,3}t^{3})\frac{1}{3!}+\cdots\\
&=&\left\{A_{1}(f)\beta_{1,1}+A_{2}(f)\beta_{1,2}\frac{1}{2!}+\cdots\right\}t\\
&&+\left\{A_{2}(f)\beta_{2,2}\frac{1}{2!}+A_{3}(f)\beta_{2,3}\frac{1}{3!}+\cdots\right\}t^{2}
+\cdots
\end{eqnarray*}
Por otro lado
\begin{eqnarray*}
\textbf{B}[\Opa(\Delta f)\cdot F]&=&
\left(
\begin{array}{cccc}
\beta_{1,1}&\beta_{1,2}&\beta_{1,3}&\cdots\\
0&\beta_{2,2}&\beta_{2,3}&\cdots\\
0&0&\beta_{3,3}&\cdots\\
\vdots&\vdots&\vdots&\ddots
\end{array}
\right)
\left(
\begin{array}{c}
\frac{A_{1}(f)}{1!}\\
\frac{A_{2}(f)}{2!}\\
\frac{A_{3}(f)}{3!}\\
\vdots
\end{array}
\right)\\
&=&(\textbf{B}_{1},\textbf{B}_{2},\textbf{B}_{3},...)[\Opa(\Delta f)\cdot F].
\end{eqnarray*}
Juntando todo obtenemos el resultado deseado.
\end{proof}

La matriz de conexi\'on del flujo $\Phi(t,x,\Delta f(x))$ es la matriz identidad infinita 
$\textbf{I}$. Si $\textbf{A}$ y $\textbf{B}$ son las matrices de conexi\'on de los flujos
$\Phi_{A}$ y $\Phi_{B}$, entonces $\textbf{BA}$ es la matriz de conexi\'on de 
$\Phi_{A}\circledcirc\Phi_{B}$. Luego el conjunto de matrices de conexi\'on forman un grupo y
tenemos el siguiente

\begin{theorem}
Denote $\mathcal{C}$ el grupo de matrices de conexi\'on de conjuntos bases. Entonces
$\mathcal{F}$ es anti-isomorfo a $\mathcal{C}$.
\end{theorem}
\begin{proof}
Sean $A,B\in\hurw_{R}[[\delta]]$ operadores delta con conjuntos base $(a_{n}(t))_{n\in\N}$ y
$(b_{n}(t))_{n\in\N}$, respectivamente y sea $\Phi_{C}=\Phi_{A}\circledcirc\Phi_{B}$
con $c_{n}(t)=a_{n}(\underline{b}(t))$. Ahora sean $a_{n}(t)=\sum_{k=0}^{n}\alpha_{k,n}t^{k}$,
$b_{n}(t)=\sum_{k=0}^{n}\beta_{k,n}t^{k}$ y $c_{n}(t)=\sum_{k=0}^{n}\alpha_{k,n}t^{k}$ con matrices
de conexi\'on $\textbf{A},\textbf{B}$ y $\textbf{C}$. Defina el mapa 
$\mu:\mathcal{F}\rightarrow\mathcal{C}$ por $\mu(\Phi_{A})=\textbf{A}$. Entonces
$\mu(\Phi_{A}\circledcirc\Phi_{B})=\mu(\Phi_{C})=\textbf{B}\textbf{A}=\mu(\Phi_{B})\mu(\Phi_{A})$.
Adem\'as $\mu(\Phi)=\textbf{I}$ y $\mu(\Phi_{A}^{-1})=\mu(\Phi_{A})^{-1}$, en donde $\Phi_{A}^{-1}$
es el flujo inverso de $\Phi_{A}$. Luego $\mu$ es un anti-isomorfismo de grupos.
\end{proof}

Finalizamos este art\'iculo mostrando ejemplos de sistemas din\'amicos tipo delta. En particular 
cuando $Q$ es alguno de los siguientes operadores: $\vartriangle$, $\triangledown$, 
$E^{\alpha}\delta$ y $\log(I+\delta)$.

\section{Sistemas din\'amicos de diferencia forward}

Un sistema din\'amico discreto uni-dimensional es uno de la forma
\begin{equation}\label{eqn_sdd}
y_{n+1}=g(y_{n})
\end{equation}
en donde $n\in\N$. Este tipo de ecuaciones son de gran importancia por sus aplicaciones 
en econom\'ia, ciencias sociales, biolog\'ia, f\'isica, ingenier\'ia, redes neurales, etc.
La ecuaci\'on (\ref{eqn_sdd}) generalmente no tiene soluci\'on exacta cuando 
$g$ es una funci\'on no lineal. Sin embargo, los resultados anteriores permiten echar a un lado
esta limitaci\'on. Usamos la diferencia forward $\vartriangle$ de $y_{n}$, esto es,
$\vartriangle y_{n}=y_{n+1}-y_{n}$. Luego (\ref{eqn_sdd}) es equivalente a la ecuaci\'on
\begin{equation}
\vartriangle y_{n}=f(y_{n})
\end{equation}
en donde $f(x)=g(x)-x$.

El operador $\vartriangle$ es un operador delta con conjunto base $q=((t)_{n})_{n\in\N}$ donde
$(t)_{n}=t(t-1)(t-2)\cdots(t-n+1)$ y
\begin{enumerate}
\item $(t)_{0}=1$,
\item $(0)_{n}=0$,
\item $\vartriangle(t)_{n}=n(t)_{n-1}$
\end{enumerate}

Sea $L$ la representaci\'on umbral de $(t)_{n}$. Tenemos

\begin{definition}
Definimos un \textbf{sistema dinámico forward} sobre el anillo $R$ como la terna 
$(R,L\hurw_{S}[[t]],\Phi_{\vartriangle})$ donde $R$ es el \textbf{conjunto de tiempos}, 
$L\hurw_{S}[[t]]$ el \textbf{espacio de fases} y $\Phi_{\vartriangle}$ es el \textbf{flujo forward} 
\begin{equation}\label{flujo_diferencial}
\Phi_{\vartriangle}(t,x,\Delta f(x))=x+\sum_{k=1}^{\infty}A_{k}([\delta^{(k-1)}f(x)])\frac{(t)_{k}}{k!}.
\end{equation}
esto es, $\Phi_{\vartriangle}$ es el mapa definido
por $\Phi_{\vartriangle}:R\times L\hurw_{S}[[t]]\times\Delta\hurw_{R}[[x]]\rightarrow L\hurw_{S}[[t]]$ donde $S=\hurw_{R}[[x]]$.
\end{definition}

Por definici\'on $\frac{(t)_{n}}{n!}=\binom{t}{n}$. Entonces 
\begin{equation}\label{eqn_flujo_for}
\Phi_{\vartriangle}(t,x,\Delta f(x))=x+\sum_{k=1}^{\infty}A_{k}([\delta^{(n-1)}f(x)])\binom{t}{k}
\end{equation}
es la soluci\'on de (\ref{eqn_sdd}) para todo $t$ en el dominio $R$. Cuando $t\in\N$, entonces (\ref{eqn_flujo_for}) se reduce a 
\begin{equation}
y_{n}=\Phi_{\vartriangle}(n,x,\Delta f(x))=x+\sum_{k=1}^{n}A_{k}([\delta^{(n-1)}f(x)])\binom{n}{k}
\end{equation}
y tenemos una soluci\'on exacta a la ecuaci\'on (\ref{eqn_sdd}) para todo 
$f(x)=g(x)-x\in\hurw_{R}[[x]]$. Es de inter\'es notar que $\Phi_{\vartriangle}$ es una funci\'on
de coeficientes binomiales. Luego es posible una interpretaci\'on combinatorial a estas soluciones.

Por un uso directo de la prueba de la Proposici\'on \ref{prop_pol_lin} obtenemos

\begin{equation}
\rho_{q}\Opa(\Delta(ax+b))=\frac{1}{a}(ax+b)\left((1+a)^{t}-1\right)
\end{equation}

y si $f(x)=\prod_{k=1}^{m}(a_{k}x+b_{k})$ entonces

\begin{equation}\label{eqn_sdd_poli}
\Phi_{\vartriangle}(t,x,\Delta f(x))=x+\bcastq_{k=1}^{m}\frac{1}{a_{k}}(a_{k}x+b_{k})\left((1+a_{k})^{t}-1\right)
\end{equation}

es soluci\'on de la ecuaci\'on en diferencias

\begin{equation}
\vartriangle y_{n}=(a_{1}y_{n}+b_{1})(a_{2}y_{n}+b_{2})\cdots(a_{m}y_{n}+b_{m}).
\end{equation}

Como caso particular tenemos la ecuaci\'on logistica. Esta es la ecuaci\'on en diferencias de orden
uno no lineal m\'as simple y hasta ahora no exist\'ia una soluci\'on en forma cerrada para ella.
Esta ecuaci\'on es encontrada en cualquier libro de texto b\'asico sobre ecuaciones de
diferencias. Esta ecuaci\'on se escribe como
\begin{equation}\label{eqn_logis}
y_{n+1}=\mu y_{n}(1-y_{n}).
\end{equation}
Para resolverla debemos resolver la ecuaci\'on equivalente $\vartriangle y_{n}=y_{n}(\mu-1-\mu y_{n})=-y_{n}(\mu y_{n}-(\mu-1))$. Aqu\'i $f(x)=-x(\mu x-(\mu-1))$. Luego
\begin{eqnarray*}
\Phi_{\vartriangle}(t,x,\Delta f(x))&=&x+\Psi_{\vartriangle}(t,x,-\Delta x)\circledast_{\vartriangle}\Psi_{\vartriangle}(t,x,-\Delta(\mu x-(\mu-1)))\\
&=&x+(-x)\circledast_{\vartriangle}[\mu x-(\mu-1)](2^{t}-1)
\end{eqnarray*}

Con esta soluci\'on notamos que
\begin{eqnarray*}
\Phi_{\vartriangle}(t,0,\Delta f(x))&=&0\\
\Phi_{\vartriangle}\left(t,\frac{\mu-1}{\mu},\Delta f(x)\right)&=&\frac{\mu-1}{\mu}
\end{eqnarray*}
As\'i los ceros de $-x(\mu x-(\mu-1))$ corresponde con los puntos fijos de $\mu x(1-x)$.

Por otro lado, la ecuaci\'on log\'istica continua $y^{\prime}=-y(\mu y-(\mu-1))$ tiene flujo
\begin{equation}
\Phi(t,x,\Delta f(x))=\frac{(\mu-1)xe^{(\mu-1)t}}{\mu x(e^{(\mu-1)t}-1)+(\mu-1)}.
\end{equation}
Luego usando la representaci\'on umbral $L(t^{n})=(t)_{n}$ obtenemos otra forma de escribir la
soluci\'on de (\ref{eqn_logis})

\begin{equation}
L\left(\frac{(\mu-1)xe^{(\mu-1)t}}{\mu x(e^{(\mu-1)t}-1)+(\mu-1)}\right)=x+(-x)\circledast_{\vartriangle}[\mu x-(\mu-1)](2^{t}-1).
\end{equation}

Otra aplicaci\'on importante de las ecuaciones de diferencias la encontramos en la teor\'ia de
fractales, especialmente en mapas polinomiales iterativos. Uno de tales mapas son los mapas 
cuadr\'aticos complejos $z_{n+1}=z_{n}^{2}+c$ \'utiles para construir conjuntos de Mandelbrot.
Por ser una ecuaci\'on en diferencias de orden uno podemos usar (\ref{eqn_sdd_poli})
para encontrar la soluci\'on exacta a dicho mapa. 

Resolver
\begin{equation}
z_{n+1}=z_{n}^{2}+c
\end{equation}

es resolver la ecuaci\'on equivalente

\begin{equation}
\vartriangle z_{n}=z_{n}^{2}-z_{n}+c
\end{equation}
Para este caso tenemos que $f(x)=x^{2}-x+c=(x-\alpha)(x-\overline{\alpha})$ en donde 
$\alpha=\frac{1+\sqrt{4c-1}i}{2}$.
Luego
\begin{eqnarray*}
z_{n}&=&\Phi_{\vartriangle}(n,x,\Delta f(x))\\
&=&x+\Psi_{\vartriangle}(n,x,\Delta(x-\alpha))\circledast_{\vartriangle}\Psi_{\vartriangle}(n,x,\Delta(x-\overline{\alpha}))\\
&=&x+[(x-\alpha)(2^{n}-1)]\circledast_{\vartriangle}[(x-\overline{\alpha})(2^{n}-1)]
\end{eqnarray*}

Igualmente como se hizo con la ecuaci\'on log\'istica tenemos
\begin{equation}
L\left(\frac{\overline{\alpha}e^{\sqrt{4c-1}it}(x-\alpha)-\alpha(x-\overline{\alpha})}{(x-\alpha)e^{\sqrt{4c-1}it}-(x-\overline{\alpha})}\right)=x+[(x-\alpha)(2^{n}-1)]\circledast_{\vartriangle}[(x-\overline{\alpha})(2^{n}-1)]
\end{equation}

\section{Sistemas din\'amicos de diferencia backward}

Si $Q$ es el operador backward $\triangledown$, entonces su conjunto base son los factoriales
rising $q=((t)^{n})_{n\in\N}$ donde $(t)^{n}=t(t+1)(t+2)\cdots(t+n-1)$ y
\begin{enumerate}
\item $(t)^{0}=1$,
\item $(0)^{n}=0$,
\item $\triangledown(t)^{n}=n(t)^{n-1}$
\end{enumerate}

Con el operador backward podemos resolver ecuaciones en diferencias de la forma

\begin{equation}
\triangledown y_{n}=f(y_{n})
\end{equation}
en donde $\triangledown y_{n}=y_{n}-y_{n-1}$. Las ecuaciones en diferencias backward se aplican en 
las mismas \'areas en donde son aplicados las ecuaciones en diferencias forward.

Sea $L(t^{n})=(t)^{n}$ la representaci\'on umbral de $(t)^{n}$. Entonces tenemos la siguiente
definici\'on

\begin{definition}
Definimos un \textbf{sistema dinámico backward} sobre el anillo $R$ como la terna 
$(R,L\hurw_{S}[[t]],\Phi_{\triangledown})$ donde $R$ es el \textbf{conjunto de tiempos}, 
$L\hurw_{S}[[t]]$ el \textbf{espacio de fases} y $\Phi_{\triangledown}$ es el \textbf{flujo backward} 
\begin{equation}\label{flujo_diferencial}
\Phi_{\triangledown}(t,x,\Delta f(x))=x+\sum_{k=1}^{\infty}A_{k}([\delta^{(k-1)}f(x)])\frac{(t)^{k}}{k!}.
\end{equation}
esto es, $\Phi_{\triangledown}$ es el mapa definido
por $\Phi_{\triangledown}:R\times L\hurw_{S}[[t]]\times\Delta\hurw_{R}[[x]]\rightarrow L\hurw_{S}[[t]]$ donde $S=\hurw_{R}[[x]]$.
\end{definition}

Por definici\'on $\frac{(t)^{k}}{k!}=\binom{t+k-1}{k}=\binomangle{t}{k}$. Entonces 
\begin{equation}
\Phi_{\triangledown}(t,x,\Delta f(x))=x+\sum_{k=1}^{\infty}A_{k}([\delta^{(n-1)}f(x)])\binomangle{t}{k}
\end{equation}
El n\'umero $\binomangle{n}{k}$ es conocido como coeficiente binomial con repeticiones y es usado
para contar multiset de cardinalidad $k$ con elementos tomados de un conjunto de $n$ elementos.

Denote $\exp(a)$ la sucesi\'on $(a^{n})_{n\in\N_{1}}$ en $H_{R}$. En [3] fue mostrado que
\begin{equation}\label{eqn_R_prod}
\Opa(a\Delta f(x))=\exp(a)\Opa(\Delta f(x)),
\end{equation}

El siguiente resultado relaciona a los sistemas din\'amicos discretos de tipo forward y backward

\begin{theorem}
Para todo $f(x)\in\hurw_{R}[[x]]$ se cumple
\begin{equation}
\Phi_{\triangledown}(t,x,\Delta f(x))=\Phi_{\vartriangle}(-t,x,-\Delta f(x)).
\end{equation}
\end{theorem}
\begin{proof}
Sabemos que $\binomangle{t}{k}=(-1)^{k}\binom{-t}{k}$. As\'i
\begin{eqnarray*}
\Phi_{\triangledown}(t,x,\Delta f(x))&=&x+\sum_{k=1}^{\infty}A_{k}([\delta^{(n-1)}f(x)])\binomangle{t}{k}\\
&=&x+\sum_{k=1}^{\infty}A_{k}([\delta^{(n-1)}f(x)])(-1)^{k}\binom{-t}{k}\\
&=&\Phi_{\vartriangle}(-t,x,-\Delta f(x))
\end{eqnarray*}
en donde hemos usado el (\ref{eqn_R_prod}) con $a=-1$.
\end{proof}

Por otro lado, podemos usar el anterior teorema y la ecuaci\'on (\ref{eqn_sdd_poli}) para mostrar que 

\begin{equation}
\Phi_{\triangledown}(t,x,\Delta f(x))=x+\bcastq_{k=1}^{m}\frac{1}{a_{k}}(a_{k}x+b_{k})((1-a_{k})^{t}-1)
\end{equation}

es soluci\'on de la ecuaci\'on en diferencia backward

\begin{equation}
\triangledown y_{n}=(a_{1}y_{n}+b_{1})(a_{2}y_{n}+b_{2})\cdots(a_{m}y_{n}+b_{m}).
\end{equation}

El teorema anterior implica que es suficiente con resolver uno de los dos tipos de ecuaciones de
diferencias y luego por aplicaci\'on de este obtener la soluci\'on del otro tipo.

\section{Sistemas din\'amicos tipo Abel}

Denote $A(\alpha)$ el operador delta $E^{\alpha}\delta$. Este operador es conocido como operador Abel
donde sus polinomios bases son los polinomios de Abel 
$(t(t-n\alpha)^{n-1})_{n\in\N}$ y
\begin{enumerate}
\item $t(t)^{-1}=1$,
\item $0(0-n\alpha)^{n-1}=0$,
\item $A(\alpha)t(t-n\alpha)^{n-1})=nt(t-(n-1)\alpha)^{n-2}$
\end{enumerate}
Los polinomios de Abel fueron estudiados en [1] en la representaci\'on de funciones. 
Queremos resolver ecuaciones de la forma

\begin{equation}
A(\alpha)\Phi=f(\Phi).
\end{equation}

Sea $L$ la representaci\'on umbral de $(t(t-n\alpha)^{n-1})_{n\in\N}$. Tenemos la 
siguiente definici\'on

\begin{definition}
Definimos un \textbf{sistema dinámico tipo Abel} sobre el anillo $R$ como la terna 
$(R,L\hurw_{S}[[t]],\Phi_{A(\alpha)})$ donde $R$ es el \textbf{conjunto de tiempos}, 
$L\hurw_{S}[[t]]$ el
\textbf{espacio de fases} y $\Phi_{A(\alpha)}$ es el \textbf{flujo Abel} 
\begin{equation}\label{flujo_diferencial}
\Phi_{A(\alpha)}(t,x,\Delta f(x))=x+\sum_{n=1}^{\infty}A_{n}([\delta^{(n-1)}f(x)])\frac{t(t-n\alpha)^{n-1}}{n!}.
\end{equation}
esto es, $\Phi_{A(\alpha)}$ es el mapa definido
por $\Phi_{A(\alpha)}:R\times L\hurw_{S}[[t]]\times\Delta\hurw_{R}[[x]]\rightarrow L\hurw_{S}[[t]]$ donde $S=\hurw_{R}[[x]]$.
\end{definition}

Usando (\ref{eqn_sdd_poli}) encontramos que 

\begin{equation}
\Phi_{A(\alpha)}(t,x,\Delta f(x))=x+\bcastq_{k=1}^{m}\frac{1}{a_{k}}(a_{k}x+b_{k})\left(e^{tW(\alpha)/\alpha}-1\right),
\end{equation}
en donde $W$ es la funci\'on $W$ de Lambert, es la soluci\'on de la ecuaci\'on
\begin{equation}
A(\alpha)\Phi=(a_{1}\Phi+b_{1})(a_{2}\Phi+b_{2})\cdots(a_{n}\Phi+b_{n})
\end{equation}

La siguiente identidad relaciona los flujos de las ecuaciones $A(\alpha)\Phi=af(\Phi)$ y
$A(a\alpha)\Phi=f(\Phi)$

\begin{theorem}
Para todo $a\in R$ se cumple
\begin{equation}
\Phi_{A(\alpha)}(t,x,a\Delta f(x))=\Phi_{A(a\alpha)}(at,x,\Delta f(x)).
\end{equation}
\end{theorem}
\begin{proof}
\begin{eqnarray*}
\Phi_{A(\alpha)}(t,x,a\Delta f(x))&=&x+\sum_{n=1}^{\infty}a^{n}A_{n}([\delta^{(n-1)}f(x)])\frac{t(t-n\alpha)^{n-1}}{n!}\\
&=&x+\sum_{n=1}^{\infty}A_{n}([\delta^{(n-1)}f(x)])\frac{(at)(at-na\alpha)^{n-1}}{n!}\\
&=&\Phi_{A(a\alpha)}(at,x,\Delta f(x)).
\end{eqnarray*}
\end{proof}

Con esta identidad podremos, quizas, darle una estructura m\'as que de grupo uni par\'ametrico
que la que se obtiene con el flujo de tipo Abel, pero no ser\'a hecho en este art\'iculo.

\section{Sistemas din\'amicos tipo Touchard}

Los polinomios de Touchard son polinomios de la forma
\begin{equation}
\phi_{n}(t)=\sum_{k=0}^{n}S(n,k)t^{k}
\end{equation}
en donde $S(n,k)$ son los n\'umeros de Stirling de segunda clase. 
Los polinomios de Touchard son polinomios base del operador delta de Touchard $T=\log(I+\delta)$ y
satisfacen
\begin{enumerate}
\item $\phi_{0}(t)=1$,
\item $\phi_{n}(0)=0$,
\item $\log(I+\delta)\phi_{n}(t)=n\phi_{n-1}(t)$.
\end{enumerate}

Sea $L$ la representaci\'on umbral de $\phi_{n}(t)$. Entonces tenemos la definici\'on

\begin{definition}
Definimos un \textbf{sistema dinámico tipo Touchard} sobre el anillo $R$ como la terna 
$(R,L\hurw_{S}[[t]],\Phi_{T})$ donde $R$ es el \textbf{conjunto de tiempos}, $L\hurw_{S}[[t]]$ el
\textbf{espacio de fases} y $\Phi_{T}$ es el \textbf{flujo Touchard} 
\begin{equation}\label{flujo_diferencial}
\Phi_{T}(t,x,\Delta f(x))=x+\sum_{n=1}^{\infty}A_{n}([\delta^{(n-1)}f(x)])\dfrac{\phi_{n}(t)}{n!}.
\end{equation}
esto es, $\Phi_{T}$ es el mapa definido
por $\Phi_{T}:R\times L\hurw_{S}[[t]]\times\Delta\hurw_{R}[[x]]\rightarrow L\hurw_{S}[[t]]$ donde 
$S=\hurw_{R}[[x]]$.
\end{definition}

Usando (\ref{eqn_sdd_poli}) encontramos que 

\begin{equation}
\Phi_{T}(t,x,\Delta f(x))=x+\bcastq_{k=1}^{m}\frac{1}{a_{k}}(a_{k}x+b_{k})\left(e^{t(e-1)}-1\right)
\end{equation}

es el flujo de Touchard de la ecuaci\'on

\begin{equation}
\log(I+\delta)\Phi=(a_{1}\Phi+b_{1})(a_{2}\Phi+b_{2})\cdots(a_{n}\Phi+b_{n})
\end{equation}

\section{Conclusi\'on}
En este art\'iculo los sistemas din\'amicos tipo delta fueron definidos y algunas de sus propiedades
estudiadas. La teor\'ia aqu\'i desarrollada es muy \'util para resolver sistemas din\'amicos tipo
delta no lineales. Particularmente los sistemas din\'amicos de diferencia y de tipo Touchard pueden
ser \'util para resolver problemas en combinatoria enumerativa. Como los polinomios de Abel
son usados en la teor\'ia algebraica estad\'istica, entonces los sistemas din\'amicos tipo Abel
podr\'ian ser \'util para aplicaciones en estad\'istica. En este art\'iculo la derivaci\'on fue 
variada. Estamos interesados en variar el dominio de integridad y dejar fijo el operador delta
y de esta forma estudiar los distintos resultados que se pueden obtener. Otra posibilidad es 
variar el factorial, esto es, construir una teor\'ia para un factorial generalizado.

\Addresses
\end{document}